\documentclass[11pt,a4paper]{article}

\usepackage[margin=1.206in]{geometry}
\usepackage[latin1]{inputenc}
\usepackage[ruled,vlined]{algorithm2e}
\usepackage{amsmath}
\usepackage{amsthm}
\usepackage{float}
\usepackage{bbm}
\usepackage{mathtools}
\usepackage{amssymb,url}
\usepackage{amsfonts}
\usepackage{marginnote}
\usepackage{mathabx} 
\usepackage{amscd}%
\usepackage{graphicx}
\usepackage{dsfont}
\usepackage{multirow}
\usepackage{enumerate}
\usepackage{float}
\usepackage{libertine}
\usepackage{bm}
\usepackage{caption}
\usepackage{subcaption}
\usepackage[dvipsnames]{xcolor}
\usepackage[citecolor=NavyBlue,colorlinks=true,linkcolor=NavyBlue,urlcolor=NavyBlue]{hyperref}
\usepackage[compact,raggedright,small]{titlesec}
\usepackage{enumitem}
\usepackage{cleveref}
\usepackage{lineno}
\modulolinenumbers[5]
\usepackage{tikz}
\usepackage{pgfplots}
\usetikzlibrary{pgfplots.groupplots}
\usetikzlibrary{plotmarks}
\usetikzlibrary{calc}
\usepackage{tikz-cd}

\usepackage{thmtools, thm-restate}
\usepackage{hyperref}

\declaretheorem{theorem}

\declaretheorem{corollary}

\newcommand\norm[1]{\left\lVert#1\right\rVert}
\newcommand{\spn}{\operatorname{span}}
\renewcommand{\mod}{\operatorname{mod}} 
\newcommand{\supp}{\operatorname{supp}} 
\newcommand{\vect}{\operatorname{vec}}
\newcommand{\tr}{\operatorname{tr}}
\newcommand{\sign}{\operatorname{sgn}}
\renewcommand{\l}{\ell}

\newtheorem{proposition}[theorem]{Proposition}

\crefformat{equation}{\textup{#2(#1)#3}}
\crefrangeformat{equation}{\textup{#3(#1)#4--#5(#2)#6}}
\crefmultiformat{equation}{\textup{#2(#1)#3}}{ and \textup{#2(#1)#3}}
{, \textup{#2(#1)#3}}{, and \textup{#2(#1)#3}}
\crefrangemultiformat{equation}{\textup{#3(#1)#4--#5(#2)#6}}%
{ and \textup{#3(#1)#4--#5(#2)#6}}{, \textup{#3(#1)#4--#5(#2)#6}}{, and \textup{#3(#1)#4--#5(#2)#6}}

\Crefformat{equation}{#2Equation~\textup{(#1)}#3}
\Crefrangeformat{equation}{Equations~\textup{#3(#1)#4--#5(#2)#6}}
\Crefmultiformat{equation}{Equations~\textup{#2(#1)#3}}{ and \textup{#2(#1)#3}}
{, \textup{#2(#1)#3}}{, and \textup{#2(#1)#3}}
\Crefrangemultiformat{equation}{Equations~\textup{#3(#1)#4--#5(#2)#6}}%
{ and \textup{#3(#1)#4--#5(#2)#6}}{, \textup{#3(#1)#4--#5(#2)#6}}{, and \textup{#3(#1)#4--#5(#2)#6}}

\makeatletter
\newcommand{\footremember}[2]{%
\footnote{#2}
\newcounter{#1}
\setcounter{#1}{\value{footnote}}%
}
\newcommand{\footrecall}[1]{%
\footnotemark[\value{#1}]%
}
\makeatother
\title{\large\bfseries Well conditioned ptychograpic imaging via lost subspace completion}
\date{\today}
\author{Anton Forstner\footremember{TUM}{Department of Mathematics, Technical University of Munich, 85748 Garching bei M\"unchen, Germany (\href{mailto:anton.forstner@tum.de}{anton.forstner@tum.de}, \href{mailto:felix.krahmer@tum.de}{felix.krahmer@tum.de}, \href{mailto:oleh.melnyk@tum.de}{oleh.melnyk@tum.de}, \href{mailto:sissouno@ma.tum.de}{sissouno@ma.tum.de})} \and Felix Krahmer\footrecall{TUM} \and Oleh Melnyk\footrecall{TUM} \footremember{HMGU}{Mathematical Imaging and Data Analysis, ICT, Helmholtz Center Munich, 85764 Neuherberg, Germany} \and Nada Sissouno\footrecall{TUM} \footrecall{HMGU} }

\begin{document}

\maketitle
\vspace{-3mm}

\begin{abstract}
Ptychography, a special case of the phase retrieval problem, is a popular method in modern imaging. Its measurements are based on the shifts of a locally supported window function. In general, direct recovery of an object from such measurements is known to be an ill-posed problem. Although for some windows the conditioning can be controlled, for a number of important cases it is not possible, for instance for Gaussian windows. In this paper we develop a subspace completion algorithm, which enables stable reconstruction for a much wider choice of windows, including Gaussian windows. The combination with a regularization technique leads to improved conditioning and better noise robustness. 
\end{abstract}
\vspace{4mm}


\section{Introduction}\label{introduction}

The aim in the phase retrieval problem is to recover a vector $x_0 \in \mathbb{C}^d$ from measurements given by
\begin{equation} \label{eq: ph ret}
y_j = \left| \langle x_0, a_j \rangle \right|^2 + n_j \text { or } y=\left| A x \right|^2 + n,
\end{equation}
where the measurement vectors $a_j$ (or the measurement matrix $A$, respectively) is known and the noise $n$ is unknown, but small. 
Note that no recovery method can distinguish the solution candidates $\{ e^{i \theta} \cdot x_0 : ~ \theta \in [0, 2 \pi] \}$, since they all generate the same set of measurements, so at best, reconstruction is possible up to a global phase factor. To account for this ambiguity, one typically measures the quality of reconstruction via   
\[
d(x,x_0) = \min_{\theta \in [0, 2 \pi]} \norm{x- e^{i \theta} x_0}_2.
\]
Phase retrieval is a central problem in various applied fields such as optical imaging \cite{Shechtman.2015}, crystallography \cite{Liu.2012} or imaging of noncrystalline materials \cite{Miao.2008}. In these fields, the aim is to draw conclusions about an object of interest based on diffraction measurements made from illuminating the object with some sort of radiation such as light waves or X-ray beams. These measurements, however, only measure the intensity of the optical wave that is reaching a detector and are not capable to measure any phase information \cite{Shechtman.2015}, although it encodes important structural information \cite{Candes.2013}. Thus, some restrictive model assumption or measurement redundancy is crucially required for recovery.

\subsection{Ptychography and locally supported measurements}
Ptychography is a specific form of a redundant phase retrieval problem that received considerable attention in recent years. Instead of a full illumination of the specimen $x_0$,  multiple measurements are obtained where in each of them the X-ray beam is focused on a small region of the object. For each focus region the intensities of the resulting far field diffraction patterns are captured by a detector. In the experimental setup, this is realized by moving the object after each measurement. To obtain redundancies, the regions are chosen to be overlapping. The whole process is summarized in Figure \ref{fig:ptychography}.

 \begin{figure}[!t] 
        \center{\includegraphics[width=0.8\textwidth]
        {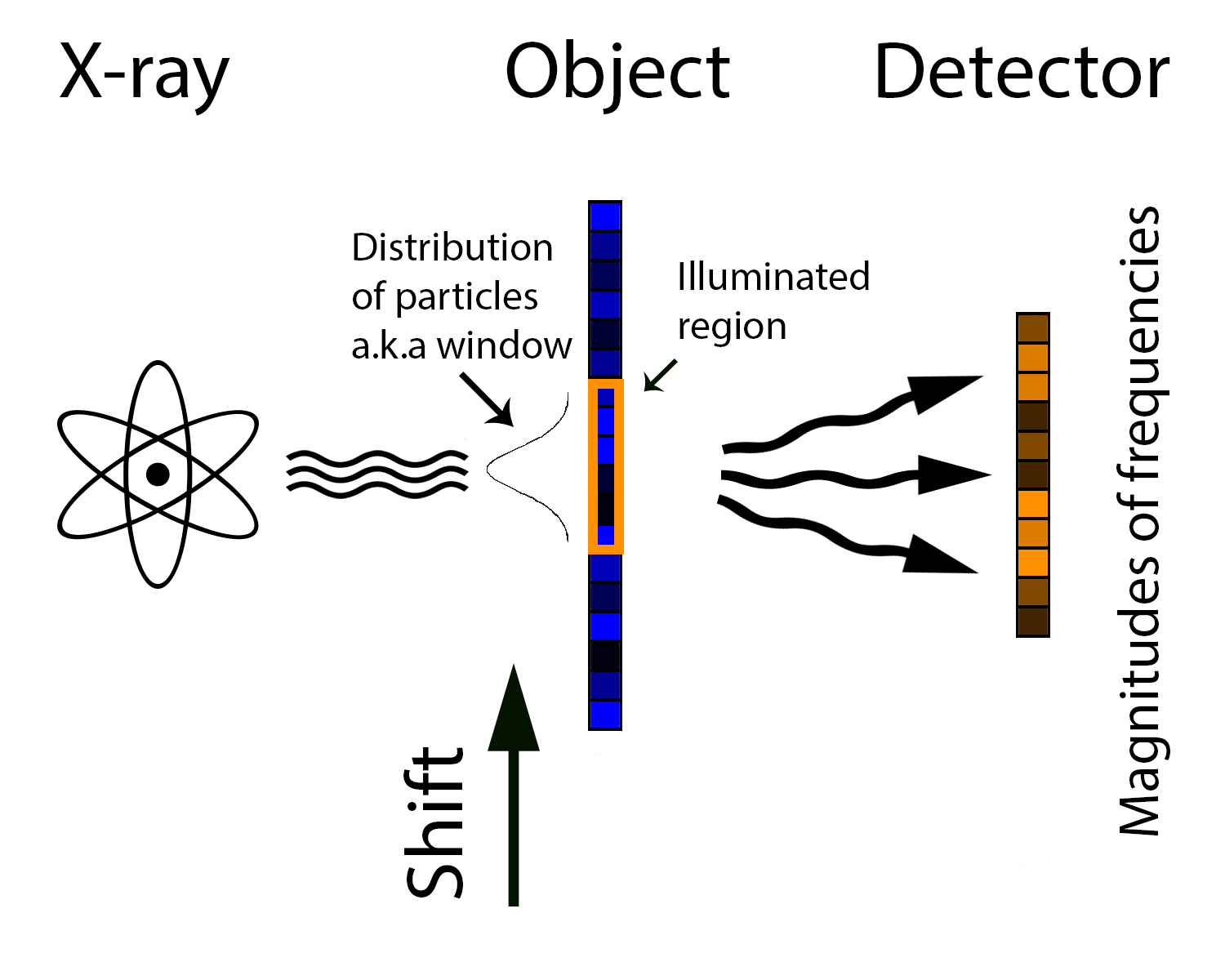}}
        \caption{A typical one-dimensional ptychographic setup}
        \label{fig:ptychography}
 \end{figure}

The history of ptychography dates back to works by Hoppe \cite{Hoppe.1969} from the late 60ths. In the following decades it was studied mainly in a few works by Rodenburg and coauthors \cite{Bates.1989,McCallum.1992,Rodenburg.1993}. Only in 2007 with the development of new imaging methods and devices, ptychography was rediscovered and became more popular, especially in the synchrotron community \cite{Holler.}.

Ptychography has applications in many different research areas across life and material sciences and has been successfully used to get image resolutions at a nanometer scale. For example it was used to image the internal structure of silk fibers \cite{Sakdinawat.2010}, the 3D pore structure of a catalyst \cite{daSilva.2015}, and to visualize stereocilia, hair cells in the inner ear, necessary for hearing and balancing that have a diameter in the nanometer range \cite{Piazza.2014}.
 
At the core of the measurement process is a detector that samples the intensities of the far field diffraction pattern  of the illuminated region. The measurements can be modeled by squared magnitudes of the Fourier coefficients of the windowed image, that is,
\begin{equation}
(y_\l)_j = \left| \displaystyle\sum_{n=1}^{d} w_n (x_0)_{n+\l} e^{- \frac{2 \pi i (j-1)(n-1)}{d}} \right|^2 + n_{j,\l} , \quad (j,\l) \in [d] \times P,
\label{def_ptychograpic_measurements}
\end{equation}
where $P \subset \{0,1,\ldots,d-1\}$ is the set of shifts of the object, $w \in \mathbb{C}^d$ is a localized illumination or window function and $n_{j,\l}$ is noise. 

\subsection{Overview of related work}

The general phase retrieval problem has been widely discussed along the time period of the last 50 years in the scientific community and a lot of research has been put in developing algorithms to tackle it efficiently. Earliest attempts are alternating projection algorithms developed in the 1970s by Gerchberg and Saxton \cite{R.W.GerchbergW.O.Saxton.1972} and Fienup \cite{Fienup.1978}. An overview and a numerical evaluation of such methods can be found in \cite{Marchesini.2007} and \cite{Marchesini.}. In general, algorithms based on alternating projection are popular among practitioners because they are easy to implement, often have low computational complexity, and produce reasonable solutions in many cases. Nevertheless, these iterative methods are mathematically difficult to analyze and their performance heavily relies on a good initial guess \cite{Marchesini.2019}. In particular, there are no global recovery guarantees for alternating projection approaches.\\

Following the emergence of compressed sensing in 2004 a number of works  aimed to analyze the phase retrieval problem from a similar viewpoint. This includes works deriving conditions for injectivity under generic measurements \cite{Balan.2006} and recovery guarantees for various algorithms under random measurement scenarios -- first for Gaussian random measurements with full randomness \cite{Candes.,Waldspurger.,Candes.2015,Alexeev.2014} and later for derandomized measurements such as subsampled spherical designs \cite{Gross.2015} and coded diffraction patterns \cite{Candes.2015,Bandeira.2014,Candes.2015b,Gross.2017}. The algorithms analyzed in these results, include convex optimization approaches such as \textit{PhaseLift} \cite{Candes.} -- these mainly show that recovery in polynomial time is possible and are not feasible for larger problem sizes -- and non-convex alternatives such as the \textit{Wirtinger Flow} algorithm \cite{Candes.2015}, which pursues a gradient descent strategy and, thus, is computationally more efficient. Up to this point, however, none of these approaches provides guarantees for ptychograpic measurements.\\

With rising interest in ptychography several methods were developed specifically for measurements of the form \eqref{def_ptychograpic_measurements}. An approach based on alternating projection methods known as PIE \cite{Faulkner.2004} is commonly used by practitioners, but again not understood mathematically. Alternatives that come with some mathematical guarantees include frame based approaches \cite{Pfander.2019}, numerical integration strategies \cite{Prusa.2017}, methods based on the properties of the Wigner distribution\cite{McCallum.1992, Chapman.1996,Perlmutter.2019}, and the so called \textit{BlockPR} algorithm \cite{Iwen.04.12.2016}.

The last mentioned method, which is also the starting point for our paper, uses deterministic measurement masks and a lifting scheme similar to \textit{PhaseLift} \cite{Candes.}. However, it is much more efficient as it exploits the locality of the measurement masks to reduce the dimension of the problem from $d^2$ in \textit{PhaseLift} to $d(2\delta -1)$ in \textit{BlockPR}. In its initial version \cite{Iwen.2016} \textit{BlockPR} was based on a greedy algorithm for recovery. Later, it was significantly improved by estimating the phases via angular synchronization \cite{Iwen.04.12.2016, Preskitt.2018} and extended to subsampled scenarios \cite{Melnyk.7820197122019, Preskitt.}. All of these works focus on rather restrictive classes of measurement windows; as we will see this is due to inherent invertibility and conditioning problems. In this paper, we address this issue proposing a strategy that allows for recovery even in the case of singularities.

\section{Preliminaries}\label{sec: preliminaries}

\subsection{Notation and Fourier transform basics}

Index sets will be expressed by the notation 
\begin{align*}
[n]_k := \left\lbrace k, \ldots, k+n-1 \right\rbrace \subset \mathbb{N}^0, \quad k \in \mathbb{N}^0,
\end{align*}
where we set ${[n] := [n]_1}$. The Hadamard product $u \circ v\in \mathbb{C}^d$  of two vectors $u,v\in \mathbb{C}^d$ is given entrywise by $(u\circ v)_j =u_j v_j$. 
Recall that for a complex vector $u \in \mathbb{C}^d$  its discrete Fourier transform (DFT) $\widehat{u} \in \mathbb{C}^d$ is given by
\begin{equation} \label{def_discrete_fourier_trans}
\widehat{u}_k = \displaystyle\sum_{n=1}^{d} u_n e^{\frac{-2 \pi i (n-1)(k-1)}{d}}.
\end{equation}

We define the discrete circular convolution of two vectors $u$ and $v$  according to \cite{Hunt.1971} as a vector $u * v$ with
\begin{align*}
(u * v)_\l = \displaystyle\sum_{k=1}^d u_{\l-k+1} v_k, \ \l \in [d].
\end{align*}
The (discrete version of the) Convolution Theorem states that the Fourier transform of the Hadamard product of two signals $u$ and $v$ in $\mathbb{C}^d$ is the (circular) convolution of the Fourier transforms of the two signals, i.e., 
\begin{equation} \label{convolution_theorem}
\widehat{u \circ v} = \widehat{u} * \widehat{v}.
\end{equation} 
A proof of this result can be found for example in \cite{Hunt.1971}. 



The (discrete) circular shift operator $S_\l: \mathbb{C}^d \rightarrow \mathbb{C}^d$, $\l \in \mathbb Z$ and the modulation operator $W_k: \mathbb{C}^d \rightarrow \mathbb{C}^d$, $k \in \mathbb Z$ are given by 
\begin{equation}\label{def_shift_operator}
\left( S_\l u\right)_j := u_{\l+j}, \quad u \in \mathbb{C}^d, \ j \in [d] 
\end{equation}
and
\begin{equation}
(W_k u)_j := e^{\frac{2 \pi i (k-1) (j-1)}{d}} u_j, \quad u \in \mathbb{C}^d, \ j \in [d].
\label{def_modulation}
\end{equation}

The next Lemma summarizes some well-known basic properties of the DFT, which are relevant for this paper. 

\begin{restatable}[Elementary Fourier transform properties]{lemma}{eftp}
\label{lem:EFTP}
Let the modulation operator $W_k$ be defined as in \eqref{def_modulation}, the DFT be defined as in \eqref{def_discrete_fourier_trans} and the Shift operator $S_\l$ be defined as in \eqref{def_shift_operator}. Then, the following properties hold for every complex vector $x \in \mathbb{C}^d$:
\begin{enumerate}
\item $\widehat{S_\l x} = W_{\l+1} \widehat{x}$ \label{p1},
\item $ \left( \widehat{\bar{x}} \right)_{k} =  \overline{ \left( \widehat{x} \right)}_{d-k+2}$ $\forall k \in [d]$. \label{conjug_in_time}
\end{enumerate}
\end{restatable}


\subsection{Model}

One diffraction measurement $(y_\l)_j$ in \eqref{def_ptychograpic_measurements} corresponds to the j-th Fourier mode of the illuminated specimen shifted by $\ell$. The window $w$ that describes the illumination is assumed to have compact support 
\begin{equation*}
\supp(w) = [\delta] \subset [d],
\end{equation*}
where $\delta$ denotes the support size.
By combining the window with the modulation in \eqref{def_ptychograpic_measurements} we obtain masks $m_j$, $j \in [d]$, given by
\begin{equation} \label{eq:pty masks}
(m_j)_n := \overline{w_n} e^{\frac{2 \pi i (j-1)(n-1)}{d}}.
\end{equation} 
With this definition the squared magnitude measurements are of the form 
\begin{equation} \label{def_y}
(y_\l)_j = \left| \langle S_\l x_0, m_j \rangle \right|^2 + n_{j,\l} = \left| \langle x_0, S_\l^* m_j \rangle \right|^2 + n_{j,\l}, \quad (j,\l) \in [K] \times P,
\end{equation} 
with $K = d$ and $m_j$ given by \eqref{eq:pty masks}. Further, we consider another set of masks, obtained by subsampling in frequency domain, i.e.,
\begin{equation} \label{def_general_mask}
  (m_j)_n =
    \begin{cases}
      \frac{1}{\sqrt[4]{2\delta-1}} \overline{w_n} \cdot e^{\frac{2 \pi i (n-1) (j-1)}{2\delta -1}} & \text{if $n \leq \delta$,}\\
      0 & \text{if $n > \delta$,}
    \end{cases},\quad
j \in [2\delta -1].
\end{equation}
In particular we are interested in a Gaussian window given by the formula
\begin{equation}\label{def_gaussian_window}
w_n = \exp \left\lbrace -\frac{\left( n-\frac{\delta+1}{2} \right)^2}{2 \sigma^2 \delta^2} \right\rbrace, \quad n \in [\delta],
\end{equation}
which is a good approximation of windows appearing in ptychography \cite{Sissouno.2019}.

\subsection{Idea of the \textit{BlockPR} algorithm}

The eigenvector-based angular synchronization \textit{BlockPR} algorithm proposed by \cite{Iwen.04.12.2016} and \cite{Preskitt.2018} uses a linear measurement operator to describe how to retrieve the measurements \eqref{def_y} from the signal $x_0$. In order to introduce this linear operator, we first observe that in the noiseless case the quadratic measurements can be lifted up and interpreted as linear measurements of the rank one matrix $x_0 x_0^*$ as done in \cite{Candes.2013}. Indeed,
\begin{equation*}
\begin{split}
(y_\l)_j 
&= \left| \langle x_0 , S_\l^* m_j \rangle \right|^2 
= \langle x_0 , S_\l^* m_j \rangle \cdot \langle S_\l^* m_j, x_0 \rangle 
= m_j^* S_\l x_0 x_0^* S_\l^* m_j  \\
&
= \tr(S_\l^* m_j m_j^* S_\l x_0 x_0^*) =  \langle x_0 x_0^*, S_\l^* m_j m_j^* S_\l \rangle_F ,
\end{split}
\end{equation*}
where $\langle \cdot , \cdot \rangle_F$ denotes the Frobenius inner product defined as $\langle A, B \rangle_F := \tr(A^* B).$ By this reformulation, the phase retrieval problem is lifted to a linear problem on the space of Hermitian $d \times d$ matrices $\mathcal{H}^d$. We consider all shifts $\l \in P = [d]_0$ and recall that the masks are compactly supported. Thus, we observe that for every matrix $A \in \spn\left\lbrace S_\l^* m_j m_j^* S_\l \right\rbrace_{j\in [K], \l \in [d]_0}$ we have $A_{k,j} = 0$ when $\left| k - j \right| < \delta$ or $|k-j| > d - \delta$. Based on this observation we introduce a family of orthogonal projection operators ${T_\delta: \mathcal{H}^d \rightarrow \mathbb{C}^{d \times d}}$, given by 
\begin{equation*}
  \left( T_\delta(A) \right)_{k,j} =
    \begin{cases}
      A_{k,j} & \text{if} \left| k - j \right| < \delta \text{ or } |k-j| > d - \delta,\\
      0, & \text{else.}\\
    \end{cases}       
\end{equation*}
We see that $T_\delta$ is the orthogonal projection operator from the space $\mathcal{H}^d$ onto range $T_\delta ( \mathcal{H}^d ) \supseteq \spn\left\lbrace S_\l^* m_j m_j^* S_\l \right\rbrace_{ j\in [K], \l \in [d]_0} $. Consequently, we can write
$$
\langle x_0 x_0^*, S_\l^* m_j m_j^* S_\l \rangle = \langle T_{\delta} (x_0 x_0^*), S_\l^* m_j m_j^* S_\l \rangle.
$$

Finally we define the linear operator $\mathcal{A}: \mathcal{H}^d \rightarrow \mathbb{C}^{D}$, $D:=K \cdot d$, describing the measurement process as
\begin{equation} \label{def_operator_Alpha}
\mathcal{A}(X) = \left[ \langle X, S_\l^* m_j m_j^* S_\l \rangle \right]_{(\l,j)}.
\end{equation}
We will denote the restriction of the operator $\mathcal{A}$ to the domain $T_{\delta}(\mathcal{H}^d)$ by $\mathcal{A}|_{T_{\delta}(\mathcal{H}^d)}$.

The operator $\mathcal{A}|_{T_{\delta}(\mathcal{H}^d)}$ allows to reformulate \eqref{def_y} in the absence of noise as
\begin{equation*} \label{linear_system_to_get_X_0}
(y_\l)_j = \left| \langle x_0 , S_\l^* m_j \rangle \right|^2 = \langle T_{\delta}(x_0 x_0^*), S_\l^* m_j m_j^* S_\l \rangle = \left(\mathcal{A}|_{T_{\delta}(\mathcal{H}^d)} T_{\delta}\left(x_0 x_0^* \right)\right)_{(\l,j)},
\end{equation*}
which we wish to invert in order to obtain the lifted and on $T_{\delta}$ projected rank 1 object
\begin{equation*}
X_0 := T_{\delta}(x_0 x_0^*).
\label{def_X0}
\end{equation*}
From this matrix we form a banded matrix $\tilde X_0$ by entrywise normalization of non-zero entries of $X_0$, i. e.,
\begin{equation}\label{norm_X}
(\tilde X_0)_{k,j} =
\begin{cases}
\frac{(X_0)_{k,j}}{|(X_0)_{k,j}|}, & (\tilde X_0)_{k,j} \neq 0, \\
0, & \text{otherwise.}
\end{cases}
\end{equation}

The magnitudes of the entries of the signal $x_0$ can be recovered as square roots of the diagonal elements of $X_0$ as in \cite{Iwen.2016}. The phases of $x_0$ can be obtained from the entrywise normalization of the top eigenvector of matrix $\tilde X_0$ as a result of solving the angular synchronization problem. The whole reconstruction of $x_0$ as proposed in \cite{Iwen.04.12.2016} is summarized in Algorithm \ref{algo:1}.\par

\medskip
\begin{algorithm}[H] 
\caption{Fast Phase Retrieval from Local Measurements}
\SetAlgoLined
\SetKwInOut{Input}{Input}
\SetKwInOut{Output}{Output}
\Input{Measurements $y \in \mathbb{R}^D$ as in \eqref{def_y}}
\Output{$x \in \mathbb{C}^d$ with $x \approx e^{-i\theta} x_0$ for some $\theta \in [0, 2\pi]$ }
1. Compute $X = \left[ \mathcal{A}|_{T_{\delta}(\mathcal{H}^d)} \right]^{-1} y \in T_{\delta}(\mathcal{H}^d)$ as an Hermitian estimate of $T_\delta(x_0 x_0^*)$.\\
2. Form the banded matrix of phases $\tilde{X} \in T_{\delta}(\mathcal{H}^d)$ defined in \eqref{norm_X}. \\
3. Compute the normalized top eigenvector of $\tilde{X}$, denoted by $\tilde{x} \in \mathbb{C}^d$ with $\norm{\tilde{x}}_2 = \sqrt{d}$.\\
4. Set $x_j = \sqrt{X_{j,j}} \cdot (\tilde{x})_j$ for all $j \in [d]$ to form $x \in \mathbb{C}^d$.
\label{algo:1}
\end{algorithm}
\medskip

When the measurements are noisy, this approach needs to be refined for stable approximate recovery. This gives rise to Block Magnitude estimation \cite{Preskitt., Iwen.04.12.2016, Preskitt.2018}, which we also use as a building block of our implementation to obtain the magnitudes.
This technique utilizes the banded structure of $X_0$ by decomposing it into smaller fixed size block matrices from which then separate magnitude estimates are made. The actual estimation is realized via calculating and averaging the top eigenvector of each single block matrix which serves as a guess of the underlying signal magnitude. 

Note that invertibility and well-conditioning of the operator $\mathcal{A}|_{T_{\delta}(\mathcal{H}^d)}$ are crucial for a proper recovery by Algorithm \ref{algo:1}. In fact, invertibility of the operator $\mathcal{A}|_{T_{\delta}(\mathcal{H}^d)}$ strongly depends on the selection of the measurement masks $m_j, j \in [K]$, which determines if the condition $T_\delta ( \mathcal{H}^d ) = \spn\left\lbrace S_\l^* m_j m_j^* S_\l \right\rbrace_{ j\in [K], \l \in [d]_0} $ holds, which is equivalent to invertibility of $\mathcal{A}|_{T_{\delta}(\mathcal{H}^d)}$.
In \cite{Iwen.2016} measurement masks 
\begin{equation}
  (m_j)_n =
    \begin{cases}
      \frac{1}{\sqrt[4]{2\delta-1}} \; e^{-n/\alpha} \cdot e^{\frac{2 \pi i (n-1) (j-1)}{2\delta -1}} & \text{if $n \leq \delta$,}\\
      0 & \text{if $n > \delta$,}\\
    \end{cases}
    \label{measurement_masks_paper}   
\end{equation}
where $\alpha \in [4, \infty)$ and $m_j \in \mathbb{C}^d, j=1,...,2\delta-1$, are proven to be a good choice. They resemble the structure of ptychographic masks \eqref{eq:pty masks} and at the same time allow for a stable inversion as proven in Lemma 2 of \cite{Iwen.2016}. 

\section{Results} \label{sec: results}

%

Under the model \eqref{measurement_masks_paper}, it is shown in Lemma 2 of \cite{Iwen.2016} that the operator $\mathcal{A}|_{T_{\delta}(\mathcal{H}^d)}$ is invertible. However, it is also shown that its minimal singular value behaves proportional to $\delta^{-1}$ which leads to severe conditioning problems as the support size increases. In addition, this analysis is specific to masks of the form \eqref{measurement_masks_paper}. First steps towards a more general choice of windows were taken in \cite{Preskitt.2018, Preskitt.}, but for many windows including Gaussian windows \eqref{def_gaussian_window}, the reconstruction remained an open problem. 
The difficulty is that for many windows, the measurement operator $\mathcal{A}|_{T_{\delta}(\mathcal{H}^d)}$ is not even invertible. An example is the following statement.



\begin{corollary}\label{theo:operator non invertible}
Consider masks of the form \eqref{def_general_mask}. Assume that $d$ is even and the window $w$ satisfies the symmetry condition
\begin{equation*}
w_n = \overline{w}_{\delta - n +1} \text{ for all } n \in [\delta].
\end{equation*}
Then, $\mathcal{A}|_{T_{\delta}(\mathcal{H}^d)}$ is not invertible.
\end{corollary}

Corollary \ref{theo:operator non invertible} is equivalent to saying that $T_\delta ( \mathcal{H}^d ) \supsetneq \spn\left\lbrace S_\l^* m_j m_j^* S_\l \right\rbrace_{ j\in [K], \l \in [d]_0} $. That is for symmetric windows there exists a nontrivial ``lost subspace'' $\mathcal{L} \subset T_\delta ( \mathcal{H}^d )$ with $T_\delta ( \mathcal{H}^d ) = \spn\left\lbrace S_\l^* m_j m_j^* S_\l \right\rbrace_{ j\in [K], \l \in [d]_0} \oplus \mathcal{L}$. The main idea of this paper is that the redundancy in $X_0$ can be utilized to complete the information about the ``lost subspace'' $\mathcal{L}$. Moreover, this technique can be extended to the ``approximately lost subspace'' corresponding to the ill-conditioned part of the operator $\mathcal{A}|_{T_{\delta}(\mathcal{H}^d)}$ and, by doing so, better robustness to noise is achieved. The resulting method, the main contribution of our work, is summarized in Algorithm \ref{algo:2}. Its fundamental ideas are presented in the remainder of this section.



\subsection{Generalized inversion step}
As discussed in Section \ref{sec: preliminaries}, the measurement operator $\mathcal A$ is a linear function of $X_0$, the restriction of $x_0 x_0^*$ to the $2\delta -1$ significant diagonals. Consequently, as observed in \cite{Iwen.2016}, by reorganizing the entries of these diagonals as a vector, i.e.,
\begin{equation}
\vect(X_0)_j:=\overline{(x_0)}_{\lceil \frac{j+\delta-1}{2\delta-1}\rceil} (x_0)_{\lceil \frac{j+\delta-1}{2\delta-1}\rceil + ((j+\delta -2) \mod(2\delta-1))-\delta +1} \quad \forall j \in [D],
\label{def_vec_X0}
\end{equation}
 one can describe the measurement process as a matrix-vector product
\[
M \vect(X_0)=y
\] 
for an appropriate matrix $M\in \mathbb{C}^{D \times D}$. Note that $M$ is obtained by deleting the columns of the matrix form of $\mathcal{A}\big|_{T_\delta(\mathcal H^d)}$ corresponding to the entries in the kernel of the projection $T_{\delta}$. Thus, to understand the invertibility and stability of the operator $\mathcal{A}\big|_{T_\delta(\mathcal H^d)}$ it is sufficient to analyze the conditioning of the matrix $M$.


We split the singular value decomposition of this matrix $M$ into two parts corresponding to all singular values above some threshold $\varepsilon$ -- these will form the matrix $\Sigma_1$ -- and all singular values less or equal to $\varepsilon$  -- these will form the matrix $\Sigma_2$. The matrices $U$ and $V$ of the left and right singular vectors are split accordingly into $U_1, U_2$ and $V_1, V_2$, respectively. That is, we obtain
\begin{equation*} \label{eq:svd_of_operator}
M = U \Sigma V^*=\left( \begin{array}{rr} U_1 & U_2 \\ \end{array}\right) \left( \begin{array}{rr} \Sigma_1 & 0 \\ 0 & \Sigma_2 \end{array}\right) \left( \begin{array}{rr} V_1 & V_2 \\ \end{array}\right)^*.
\end{equation*}
The subspace $\mathcal{S}$ is defined as  
\[
\mathcal S := \spn\{ V^{(q)} : V^{(q)} \text{ column of } V_1 \}
\]

As it was shown in Corollary \ref{theo:operator non invertible}, the inverse of operator $\mathcal{A}|_{T_{\delta}(\mathcal{H}^d)}$ and hence of $M$ does not exist. To account for this and possible ill-conditioning we will work with a regularized inverse, which for $\varepsilon=0$ agrees with the Moore-Penrose pseudoinverse and which is defined as

\begin{equation}
M_{\mathcal{S}}^{-1} := \left( \begin{array}{rr} V_1 & V_2 \\ \end{array}\right) \left( \begin{array}{rr} \Sigma_1^{-1} & 0 \\ 0 & 0 \end{array}\right) \left( \begin{array}{rr} U_1 & U_2 \\ \end{array}\right)^*.
\label{svd_A|_S^{-1}}
\end{equation}

For $\varepsilon>0$ this thresholding yields a more noise robust version of the pseudoinverse operator. Note that in all cases one has
\begin{equation}\label{eq: pseudo projection}
M_{\mathcal{S}}^{-1} M = V_1 V_1^*.
\end{equation}
By constructions, this operation yields a matrix in $T_\delta(\mathcal H^d)$ in vectorized form. Combined with an embedding step that maps this vectorization back to matrix form, we obtain a regularized inverse operator, which we denote by $A\big|_{\mathcal{S}}^{-1}$.

The strategy of our approach will be to first reconstruct the well-conditioned portion of the signal $X_{\mathcal{S}} := \mathcal A\big|_{\mathcal{S}}^{-1} y $ and then use the underlying structure to recover $X_0$ from $X_\mathcal S$.  

\subsection{Subspace Completion Algorithm}
The matrix $X_0$ can be expressed through its diagonals $L^r$ , $r \in [2 \delta -1]$, which are defined componentwise as
\begin{equation}\label{def2_L_k}
\begin{split}
L^r_z =
    \begin{cases}
      \overline{(x_0)}_z (x_0)_{z+r-1} &\text{ if } 1\leq r \leq \delta, \\
      \overline{(x_0)}_{z+1} (x_0)_{z+1+r-2\delta} &\text{ if } \delta +1 \leq r \leq 2\delta -1, \\
    \end{cases}
 \end{split}
\end{equation}
where $z \in [d]$. 

As a matter of fact, each of the columns of $V$, that is, each right singular vector of $M$, is exclusively supported on a single diagonal. Even stronger, each Fourier coefficient of a diagonal $L^r$ can be computed using just one of the singular vectors, as given by the following theorem.

\begin{restatable}{theorem}{theodiag} \label{theo:diagonal_vec_relation}
Consider measurement masks of the form \eqref{def_general_mask}. Then,
\begin{equation} \label{eq:theo_diagonal_vec_relation}
\widehat{L^r}_{\xi} = \sqrt{d} \; \langle \vect(X_0), V^{(q)} \rangle,
\end{equation}
where $q$ is decomposed as 
\begin{equation*} \label{eq:indexing columns}
q = (2\delta -1)(\xi-1)+r
\end{equation*}
with $\xi \in [d]$ and $r \in [2 \delta -1]$ and $\widehat a$ denotes the discrete Fourier transform (DFT) of the vector $a \in \mathbb{C}^d$ as in \eqref{def_discrete_fourier_trans}.
\end{restatable}

Note that equality \eqref{eq:theo_diagonal_vec_relation} will not change if we replace $X_0$ by an orthogonal projection $X_\mathcal{S}$ onto a truncated singular space, as long as $V^{(q)} \in \mathcal S$.  
Consequently, one obtains

\begin{restatable}{corollary}{corLk}\label{cor:step2}
Under assumption of Theorem \ref{theo:diagonal_vec_relation} in the noiseless case where $X_\mathcal{S}= \mathcal{A}|_{\mathcal{S}}^{-1}y $ it holds that
\begin{equation*}
\sqrt{d} \; \langle  \vect(X_\mathcal{S}),V^{(q)} \rangle = 
\begin{cases}
 \widehat{L^r}_{\xi} & \text{ if } V^{(q)} \text{ is a column of } V_1, \\
0 & \text{ if } V^{(q)} \text{ is a column of } V_2.
\end{cases}
\end{equation*}
\end{restatable}

Thus, completion of the lost subspace information is equivalent to finding the missing Fourier coefficients of the diagonals $L^r$. For that, we exploit the redundancy in $X_0$ as a projection of a rank one matrix via the following lemma.  

\begin{restatable}{lemma}{lemmaquad}\label{lemma_represent_Ll_by_Lk}
Let $L^r$ and $L^{\l}$, $r,\l \in [\delta]$ be diagonals as defined in \eqref{def2_L_k}. Then, the following relation holds
\begin{equation} \label{assertion_1_lemma_represent_Ll_by_Lk}
L^r \circ S_{\l-1} (L^r)^* = L^{\l} \circ S_{r-1} (L^{\l})^*.
\end{equation}
\end{restatable}


This allows us to summarize the lost subspace completion in Algorithm \ref{algo:2}.

\begin{algorithm}[!h]
\SetKwInOut{Input}{Input}
\SetKwInOut{Output}{Output}
\Input{Measurements $y \in \mathbb{R}^D$ as in \eqref{def_y}, truncation parameter $\varepsilon \ge 0$}
\Output{$x \in \mathbb{C}^d$ with $x \approx e^{-i\theta} x_0$ for some $\theta \in [0, 2\pi]$ }
1. Compute $X_{\mathcal{S}} = \mathcal{A}|_{\mathcal{S}}^{-1} y  \in T_{\delta}(\mathcal{H}^d)$. \\
2. Evaluate Fourier coefficients of diagonals $L^r$ corresponding to columns of $V_1$ by \\
\hspace{0.3cm} Corollary \ref{cor:step2}. \\
3. Recover missing entries of $\widehat{L^r}$ for columns of $V_2$ by solving  \eqref{assertion_1_lemma_represent_Ll_by_Lk}.\\
4. Form $X$ with diagonals $L^r$.\\
5. Form the banded matrix of phases $\tilde{X} \in T_{\delta}(\mathcal{H}^d)$ defined in \eqref{norm_X}. \\
6. Compute the normalized top eigenvector of $\tilde{X}$, denoted by $\tilde{x} \in \mathbb{C}^d$ with $\norm{\tilde{x}}_2 = \sqrt{d}$.\\
7. Set $x_j = \sqrt{X_{j,j}} \cdot (\tilde{x})_j$ for all $j \in [d]$ to form $x \in \mathbb{C}^d$.
\caption{Fast Phase Retrieval from Locally Supported Measurements with Subspace Completion}
\label{algo:2}
\end{algorithm}


In general equation \eqref{assertion_1_lemma_represent_Ll_by_Lk} establishes a quadratic system. Arguably however, it is a common case that at most one Fourier coefficient of each diagonal is missing and then \eqref{assertion_1_lemma_represent_Ll_by_Lk} becomes a linear system. Indeed, the singular values corresponding to a single diagonal $L^r$ are given by 
\[
\left| \sum_{\ell=1}^{\delta -r +1} w_\ell^* w_{\ell + r-1} e^{ \frac{2 \pi i (k-1) (\ell-1)}{d} } \right|, \ k \in [d].
\]
Since this expression as a function of $k$ and extended to the full interval $[1,d]$ is highly oscillating, hitting zero at an integer point is a rare event, except for a zero that results from the symmetry for $k = d/2+1$ on some diagonals. Thus, encountering two zeros on one diagonal is much less common than just one. For the same reason it is commonly the case that on at least one of the diagonals one does not encounter any zeros. This heuristic leads to the assumptions of the following theorem, which are indeed sufficient to guarantee recovery.


\begin{restatable}{theorem}{theolin} \label{theo:linear system}
Assume that
\begin{itemize}
\item[\textbf{(A1)}] At least one diagonal $L^r$ is fully recovered in Step 2 of Algorithm \ref{algo:2}. \label{assump_1}
\item[\textbf{(A2)}] At most one Fourier coefficient of each diagonal is missing. \label{assump_2}
\end{itemize}
Then \eqref{assertion_1_lemma_represent_Ll_by_Lk} can be expressed as a linear system and Algorithm \ref{algo:2} recovers the lost Fourier coefficients via standard solution strategies.
\end{restatable}

The numerical experiments presented in the following section confirm that assumptions \textbf{(A1)} and \textbf{(A2)} are justified in realistic scenarios such as for Gaussian windows, and hence Algorithm \ref{algo:2} yields improved reconstruction quality.

\section{Numerics} \label{sec: numerics}

\begin{figure}[!b]
\begin{center}
\textbf{Robustness to Additive Gaussian Noise}\par\medskip
\end{center}
\begin{tikzpicture}[scale=0.89]

\pgfplotsset{
    scaled y ticks = false,
    axis on top,
    xtick = data,
    xticklabel style={text width=2em,align=center},
    xminorticks=true,
    yminorticks=true,
    ylabel shift={-1.5em},
    ylabel style={align=center}
}
    \begin{groupplot}[ 
        group style={
        group size=2 by 2,
        vertical sep=25pt,
        horizontal sep=35pt
        },
    ]

    \nextgroupplot[		
    		 ymode=log,
            ylabel={Reconstruction Error (mean)},
            title={a) $ d = 64$},
            cycle list name=exotic,
            xlabel={Noise level in SNR},
    ]
    \addplot  table[x = snr, y =yvalue_1,col sep=space]{pictures/datasets/custom_d64_delta8_BlockPR.csv};
    \addplot  table[x = snr, y =yvalue_2,col sep=space]{pictures/datasets/custom_d64_delta8_BlockPR.csv};
    
    \coordinate (c1) at (rel axis cs:0,1);

    \nextgroupplot[
            title={b) $d = 63$},
            ymode=log,
            ylabel={},
            cycle list name=exotic,
            xlabel={Noise level in SNR},
            legend style={at={($(0,0)$)},legend columns=3,fill=none,draw=black,anchor=center,align=center},
            legend to name= leg1
    ]
    \addplot  table[x = snr, y =yvalue_1,col sep=space]{pictures/datasets/custom_d63_delta8_BlockPR.csv};
    \addplot  table[x = snr, y =yvalue_2,col sep=space]{pictures/datasets/custom_d63_delta8_BlockPR.csv};
    \legend{
\texttt{BlockPR + SC$_{0}$},
\texttt{BlockPR},
}
    
    \coordinate (c2) at (rel axis cs:-0.1,0);
    \end{groupplot}
    \node[below] at (c2 |- current bounding box.south)
      {\pgfplotslegendfromname{leg1}};
\end{tikzpicture}%

\caption[Comparison of the reconstruction accuracy between the \textit{BlockPR} algorithm and the \textit{BlockPR + SC}$_{0}$ algorithm]{Comparison of the reconstruction accuracy between the \textit{BlockPR} algorithm and the \textit{BlockPR + SC}$_{0}$ algorithm for window size $\delta = 8$ and dimension $d=64$ (a) and $d=63$ (b).}
\label{fig:numerics_1}
\end{figure}
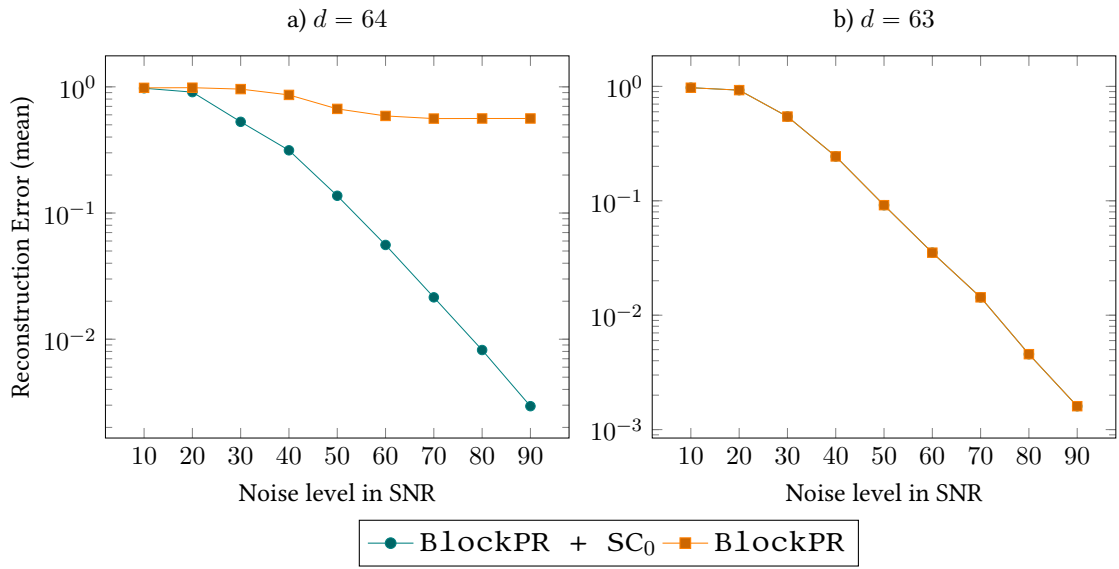

In this section we present numerical trials to asses the performance of Algorithm \ref{algo:2}. Our main objective is to demonstrate the positive effect of the additional recovery step between  Steps 1 and 2 of Algorithm \ref{algo:1}. In the following, we will denote by \textit{BlockPR} the  algorithm without that additional step as originally proposed in \cite{Iwen.04.12.2016} and summarized in Algorithm \ref{algo:1} above. As discussed in Section \ref{sec: results}, the subspace completion step is not only useful when information gets lost due to zero singular values, but also in case of very small singular values that can give rise to instabilities. In that case, Algorithm \ref{algo:2} is applied after deleting all the information which corresponds to singular values of the measurement operator $\mathcal{A}|_{T_{\delta}(\mathcal{H}^d)}$ below some threshold $\varepsilon \ge 0$. We will refer to this combined procedure by  \textit{\mbox{BlockPR + SC}}$_\varepsilon$, where $\varepsilon$ indicates the truncation level.\\ 

We will discuss three different examples in this section. In the first two examples we consider the Gaussian window \eqref{def_gaussian_window} with $\sigma = 0.3$ and in the last example the exponential window as in \eqref{measurement_masks_paper} with $\alpha=1$. 

We first illustrate the guarantees of Theorem \ref{theo:linear system}. Indeed, Gaussian windows in even dimensions fulfill the assumptions of Corollary \ref{theo:operator non invertible} and also of Theorem \ref{theo:linear system}, which indicates that \textit{\mbox{BlockPR + SC}}$_{0}$ should significantly outperform the original \textit{BlockPR} algorithm. Figure \ref{fig:numerics_1} demonstrates that this is the case.

In the second line of simulations we compare the reconstruction accuracy achieved for different choices of truncation parameters $\varepsilon$. Figures \ref{fig:numerics_2} and \ref{fig:numerics_3} show that larger truncation thresholds yield better results for larger noise levels and smaller thresholds are better suited for smaller noise levels. For small noise our approach significantly outperforms the competitor algorithm \textit{Wirtinger Flow}.



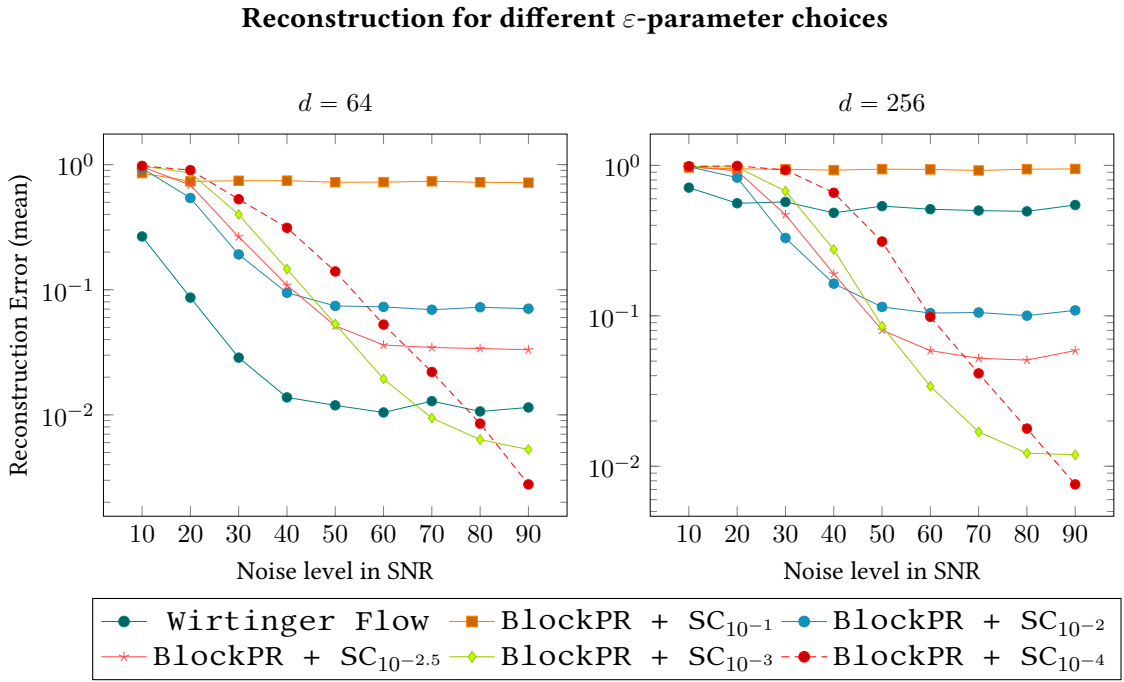
\begin{figure}[!b]
\begin{center}
\textbf{Reconstruction for different $\varepsilon$-parameter choices}\par\medskip
\end{center}
\begin{tikzpicture}[scale=0.89]

\pgfplotsset{
    scaled y ticks = false,
    axis on top,
    xtick = data,
    xticklabel style={text width=2em,align=center},
    xminorticks=true,
    yminorticks=true,
    ylabel shift={-1.5em},
    ylabel style={align=center}
}
    \begin{groupplot}[ 
        group style={
        group size=2 by 2,
        vertical sep=25pt,
        horizontal sep=35pt
        },
    ]

    \nextgroupplot[		
    		 ymode=log,
            ylabel={Reconstruction Error (mean)},
            title={$d = 64$},
            cycle list name=exotic,
            xlabel={Noise level in SNR},
    ]
    \addplot  table[x = snr, y =yvalue_1,col sep=space]{pictures/datasets/custom_d64_delta8.csv};
    \addplot  table[x = snr, y =yvalue_6,col sep=space]{pictures/datasets/custom_d64_delta8.csv};
    \addplot  table[x = snr, y =yvalue_5,col sep=space]{pictures/datasets/custom_d64_delta8.csv};
    \addplot  table[x = snr, y =yvalue_4,col sep=space]{pictures/datasets/custom_d64_delta8.csv};
    \addplot  table[x = snr, y =yvalue_3,col sep=space]{pictures/datasets/custom_d64_delta8.csv};
    \addplot  table[x = snr, y =yvalue_2,col sep=space]{pictures/datasets/custom_d64_delta8.csv};
    
    \coordinate (c1) at (rel axis cs:0,1);

    \nextgroupplot[
            title={$d = 256$},
            ymode=log,
            ylabel={},
            cycle list name=exotic,
            xlabel={Noise level in SNR},
            legend style={at={($(0,0)$)},legend columns=3,fill=none,draw=black,anchor=center,align=center},
            legend to name= leg1
    ]
    \addplot  table[x = snr, y =yvalue_1,col sep=space]{pictures/datasets/custom_d256_delta8.csv};
    \addplot  table[x = snr, y =yvalue_6,col sep=space]{pictures/datasets/custom_d256_delta8.csv};
    \addplot  table[x = snr, y =yvalue_5,col sep=space]{pictures/datasets/custom_d256_delta8.csv};
    \addplot  table[x = snr, y =yvalue_4,col sep=space]{pictures/datasets/custom_d256_delta8.csv};
    \addplot  table[x = snr, y =yvalue_3,col sep=space]{pictures/datasets/custom_d256_delta8.csv};
    \addplot  table[x = snr, y =yvalue_2,col sep=space]{pictures/datasets/custom_d256_delta8.csv};
    \legend{
\texttt{Wirtinger Flow},
\texttt{BlockPR + SC$_{10^{-1}}$},
\texttt{BlockPR + SC$_{10^{-2}}$},
\texttt{BlockPR + SC$_{10^{-2.5}}$},
\texttt{BlockPR + SC$_{10^{-3}}$},
\texttt{BlockPR + SC$_{10^{-4}}$},
}
    
    \coordinate (c2) at (rel axis cs:-0.1,0);
    \end{groupplot}
    \node[below] at (c2 |- current bounding box.south)
      {\pgfplotslegendfromname{leg1}};
\end{tikzpicture}%

\caption[Comparison of different truncation parameters $\varepsilon$ and \textit{Wirtinger Flow}]{Comparison of the reconstruction accuracy for different truncation parameters $\varepsilon$ for the \textit{BlockPR + SC}$_\varepsilon$ algorithm and \textit{Wirtinger Flow} for window size $\delta = 8$ and dimension $d=64$ (left) and $d=256$ (right).}
\label{fig:numerics_2}
\end{figure}

More specifically, we perform the following numerical experiments. For synthetic signal simulation we use i.i.d. zero-mean complex random multivariate Gaussian vectors. To model noisy data we add random Gaussian noise to our measurements. The signal to noise ratios (SNRs) will be measured in decibels (dB), that is, we consider
\begin{equation*} \label{def_snr_level}
\text{SNR (dB)} = 10 \log_{10} \Bigg( \frac{\sum_{\l = 0}^{d-1} \sum_{j=1}^{2\delta -1} (y_\l)_j^2 }{d(2\delta-1) \nu^2}\Bigg),
\end{equation*}
where $(y_\l)_j$ are measurements as in \eqref{def_y} and $\nu^2$ denotes the variance of  the Gaussian noise. Unless otherwise stated, we consider the relative error between the true underlying signal $x_0 \in \mathbb{C}^d$ and its estimate $x \in \mathbb{C}^d$  up to a global phase, i.e.
\begin{equation*}\label{def_recon_error}
\min_{\theta \in [0, 2 \pi]} \frac{\norm{x- e^{i \theta} x_0}_2}{\norm{x_0}_2}.
\end{equation*}
For more representative results each data point of the following Figures is the average of reconstructing 100 different test signals.\\

We begin with a first experiment demonstrating the improved reconstruction accuracy of the \textit{BlockPR + SC}$_{0}$ algorithm over the \textit{BlockPR} algorithm for the masks of the form \eqref{def_general_mask} with Gaussian window \eqref{def_gaussian_window}. Figure \ref{fig:numerics_1} a) shows the result of reconstructing a test signal of dimension $d=64$ for a window size $\delta = 8$ and different noise levels. As predicted by Corollary \ref{theo:operator non invertible}, the \textit{BlockPR} algorithm yields a large reconstruction error due to zero singular values resulting from the symmetry of the window. The reconstruction error of the \textit{BlockPR+SC$_0$} algorithm, in contrast, shows significant decay for decreasing noise levels, similar to odd dimensions, where no singularities arise (see Figure \ref{fig:numerics_1} b) for an example, where the results of \textit{BlockPR} and \textit{BlockPR+SC$_0$} agree as there are no zero singular values. This decay is in line with Theorem \ref{theo:linear system}. Indeed, Table~\ref{tab: sing} shows that there are only $\delta -1 = 7$ zero singular values. So, by Corollary \ref{theo:missing entries}, there is at most one zero singular value per diagonal. That is, assumptions \textbf{(A1)} and \textbf{(A2)} are satisfied, which explains the good performance  of  \textit{BlockPR+SC$_0$}. 

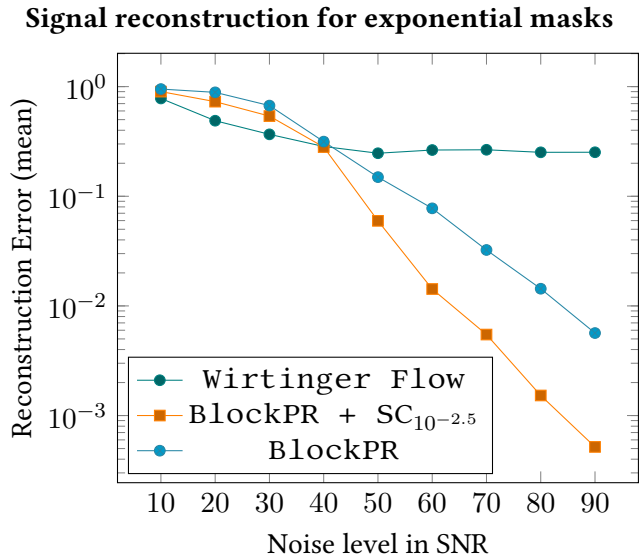
\begin{figure}[!b]
\begin{center}
\textbf{Signal reconstruction for exponential masks}\par\medskip

\begin{tikzpicture}[scale=1, transform shape]
\pgfplotstableread{pictures/datasets/expdecay_d64_delta8.csv}{\linea}
\begin{semilogyaxis}[
xlabel={Noise level in SNR},
ylabel={Reconstruction Error (mean)}, xtick=data,
legend style={at={(0.02,0.02)},anchor=south west} 
,cycle list name=exotic]

\addplot table[x=snr, y=yvalue_1]{\linea};
\addplot table[x=snr, y=yvalue_2]{\linea};
\addplot table[x=snr, y=yvalue_3]{\linea};

\legend{
\texttt{Wirtinger Flow},
\texttt{BlockPR + SC$_{10^{-2.5}}$},
\texttt{BlockPR},
}
\end{semilogyaxis}

\end{tikzpicture}
\end{center}
\caption[]{Reconstruction of a $d=64$ dimensional signal with exponential masks for \textit{BlockPR}, \textit{BlockPR + SC}$_{10^{-2.5}}$ and \textit{Wirtinger Flow} }
\label{fig:numerics_3}
\end{figure}


\begin{table}[!t] 
\centering
\caption{ Distribution of singular values for $d =64$ and $\delta = 8$ for Gaussian window.}
\label{tab: sing}
\begin{tabular}{|c|c|c|c|c|c|c|} 
\hline
Total & $ \{0\} $ & $(10^{-4}, 10^{-3}]$ & $(10^{-3}, 10^{-2.5}]$ & $(10^{-2.5}, 0.01]$ & $(0.01, 0.1]$ & $(0.1, 10]$  \\
\hline
960 & 7 &  8  & 22 & 50 & 578 & 295 \\
\hline
\end{tabular}
\end{table}


In the remainder of the section we analyze the effect of different truncation levels. As we observed in Figure \ref{fig:numerics_1}, without truncation ($\varepsilon=0$) Algorithm \ref{algo:2} struggles with high noise. The reason behind the suboptimal reconstruction quality is due to the large condition number of the operator $\mathcal{A}|_{\mathcal{S}}^{-1}$, whose application is at the core of the \textit{BlockPR} algorithm. With the subspace completion technique at hand we improve the conditioning by disregarding information corresponding to the small singular values and completing it after the inversion step.

The effect of different choices of $\varepsilon$ is illustrated in Figure \ref{fig:numerics_2} for dimensions $d=64$ and $d=256$. For comparison, we also include the \textit{Wirtinger Flow} algorithm  \cite{Candes.2015}. 

We observe that for increased truncation parameters $\varepsilon$ reconstruction at low and medium SNRs improves significantly. However, the error does not converge to zero anymore when noise diminishes as assumptions \textbf{(A1)} and \textbf{(A2)} are no longer universally satisfied. In this case one can no longer solve a linear system; rather, our implementation proceeds sequentially treating all unknown values as zero. This is of course only a heuristics and not covered by our theory. In fact, large truncation thresholds lead to a significant loss of information -- e. g., for $d = 64$ and $\varepsilon = 10^{-1}$ Table \ref{tab: sing} shows that $2/3$ of singular values are deleted, so complete reconstruction is no longer possible even for very low noise. Thus, the optimal choice of $\varepsilon$ depends on the noise level and should be based on a trade-off between noise robustness for low SNRs and perfect reconstruction for high SNRs.

Increasing the dimension from $d=64$ to $d = 256$ leads to slightly worse recovery with Algorithm \ref{algo:2}. However, the impact of the problem size on the performance of \textit{Wirtinger Flow} is much stronger. While for $d=64$ it clearly outperforms Algorithm \ref{algo:2} except for very low noise levels, for dimension $d=256$ this is the case only for rather large noise. 

We point out that the good performance of Algorithm \ref{algo:2} is not limited to Gaussian windows.  In our final example, we show that the subspace completion technique can be used to increase the set of parameters for exponential masks \eqref{measurement_masks_paper}. In this example we use $\alpha =1$, which is beyond the range of stable invertibility. Figure \ref{fig:numerics_3} illustrates the performance of Algorithms \ref{algo:1} and \ref{algo:2} together with \textit{Wirtinger Flow} for such masks. Although \textit{BlockPR} outperforms \textit{Wirtinger Flow} for moderate noise levels, we improve the accuracy even further by the application of Algorithm \ref{algo:2} with regularization parameter $\varepsilon = 10^{-2.5}$.


\section{Proofs}


\subsection{Singular values of measurement operator and proof of Corollary \ref{theo:operator non invertible}}

As discussed in Section \ref{sec: results}, the action of the measurement operator $\mathcal A$ can be described as a matrix-vector product
\[
M \vect(X_0)=y. 
\] 


As shown in Sections 2.2 and 4.4 of \cite{Iwen.2016}, the matrix $M$ takes the form
\begin{equation*}
M = \left( \begin{array}{rrrrrrrrr} M_1 & M_2 & \cdots & M_\delta & 0 & 0 & \cdots & 0 \\ 0 & M_1 &  \cdots & M_{\delta -1} & M_\delta & 0 & \cdots & 0 \\ \vdots & \vdots  & \ddots & \vdots & \vdots & \vdots & \ddots & \vdots \\ M_2 & M_3 & \cdots & 0 & 0 & 0 & \cdots & M_1 \end{array}\right),
\end{equation*}
where $M_\l$, $\l \in [\delta]$, are $(2\delta -1) \times (2\delta -1)$ matrices given by
\begin{align*}
\left(M_\l \right)_{n,j} &= 
    \begin{cases}
      \frac{\overline{w}_\l w_{j+\l-1}}{\sqrt{2\delta-1}} \cdot e^{-\frac{2 \pi i (n-1) (j-1)}{2\delta-1}} & \text{if $1 \leq j \leq \delta -\l +1$,}\\
      0 & \text{if $\delta - \l +2  \leq j \leq 2\delta - \l -1,$}\\      
      \frac{\overline{w}_{\l+1} w_{\l+j-2\delta+1}}{\sqrt{2\delta-1}} \cdot e^{-\frac{2\pi i (n-1) (j-2\delta)}{2\delta-1}} & \text{if $ 2\delta - \l \leq j \leq 2\delta -1, \l < \delta,$} \\
      0 & \text{if $j>1,$ and $\l =\delta$},\\
    \end{cases}
\end{align*}


As shown in \cite{Iwen.2016}, $M$ can be block diagonalized using the unitary block Fourier matrix $U_{2\delta-1} \in \mathbb{C}^{D \times D}$ defined as 

\begin{equation*}\label{def_u}
U_{2\delta -1} := \frac{1}{\sqrt{d}}\left (
\begin{array}{rrrr} I_{2\delta-1} & I_{2\delta-1} & \cdots & I_{2\delta-1} \\ I_{2\delta-1} & I_{2\delta-1} e^{\frac{2\pi i}{d}} & \cdots & I_{2\delta-1} e^{\frac{2\pi i (d-1)}{d}} \\  &  & \ddots &  \\ I_{2\delta-1} & I_{2\delta-1} e^{\frac{2 \pi i (d-2)}{d}} & \cdots & I_{2\delta-1} e^{\frac{2 \pi i (d-2)(d-1)}{d}} \\ I_{2\delta-1} & I_{2\delta-1} e^{\frac{2 \pi i (d-1)}{d}} & \cdots & I_{2\delta-1} e^{\frac{2 \pi i (d-1)(d-1)}{d}} \\ \end{array} \right),
\end{equation*}
where $I_{2\delta-1}$ denotes an ${(2\delta-1)} \times {(2\delta-1)}$ identity matrix. More precisely,
\begin{equation*}
M = U_{2\delta-1} J U_{2\delta-1}^*,
\label{def_rep_M}
\end{equation*}
where 
$J \in \mathbb{C}^{D \times D}$ is a block diagonal matrix of the form
\begin{equation*}
J := 
\left( \begin{array}{rrrrr} 
J_1 & 0 & \cdots & 0 & 0 \\ 
0 & J_2 & \cdots & 0 & 0 \\  
\vdots & \vdots  & \ddots & \vdots & \vdots \\ 
0 & 0 & \cdots & J_{d-1} & 0 \\
0 & 0 & \cdots & 0 & J_d 
\end{array}\right) ,
\end{equation*}
with blocks
\begin{equation*}
J_k := \displaystyle\sum_{\l=1}^{\delta} M_\l \cdot e^{\frac{2 \pi i (k-1) (\l-1)}{d}}, k \in [d].
\end{equation*}

A single block matrix $J_k \in \mathbb{C}^{(2 \delta-1) \times (2 \delta -1)}$ can be expressed as
\begin{equation}\label{def_J_k}
J_k = F_{2 \delta -1} \cdot \left( 
\begin{array}{rrrr}
z_{k,1} & 0 & \cdots & 0 \\ 
0 & z_{k,2} & 0 & \vdots \\ 
0 & 0 & \ddots & 0 \\ 
0 & \cdots & 0 & z_{k,2\delta-1}
\end{array} \right), 
\end{equation}
where $F_{2 \delta -1}$ denotes a discrete Fourier transform matrix of size $\mathbb{C}^{(2 \delta -1) \times (2 \delta -1)}$. 
As $U_{2\delta-1}$ and $F_{2\delta-1}$ are unitary matrices, the absolute values of the $z_{k,j}$'s are the singular values of $M$. 
For masks of the form \eqref{measurement_masks_paper} upper bounds on the condition number of these matrices have been shown in \cite{Iwen.2016}; they guarantee the stable invertibility of the operator $\mathcal{A}|_{T_{\delta}(\mathcal{H}^d)}$ and matrix $M$ respectively. The next proposition, however, shows that this is specific to exponentially decaying masks and that for a large class of masks naturally appearing in ptychography the operator $\mathcal{A}|_{T_{\delta}(\mathcal{H}^d)} $ is not only ill-conditioned but even singular.

\begin{proposition} \label{theorem_0_eigenvalues}
Consider masks of the form \eqref{def_general_mask}. Assume that $d$ is even and the window $w$ satisfies the symmetry condition
\begin{equation*}
w_n  = \overline{w}_{\delta - n +1} \text{ for all } n \in [\delta].
\end{equation*}
Then it holds that
\begin{align}\label{eq: theo_sing_val_zeros}
z_{\frac{d}{2}+1,j} = 0 \text{ if }\begin{cases}
       \text{ $\delta-j$ is odd where $j \in [\delta]$, } \\
      \text{ $\delta-j$ is even where $j \in [\delta-1]_{\delta+1}$ .}\\
    \end{cases}
\end{align}
\end{proposition}

To fully understand the implications of Proposition \ref{theorem_0_eigenvalues} we note that the singular value decomposition of the matrix $M$ 
takes the form
\begin{equation} \label{def_svd_fourier_measurements}
M = U_{2\delta-1} J U_{2\delta-1}^* = \underbrace{U_{2\delta-1}  
\left( \begin{array}{rrr}
F_{2 \delta -1} & & 0 \\  
& \ddots &  \\ 
0 & & F_{2\delta -1}
\end{array} \right) \sign \left( Z \right)}_{U} \underbrace{\left| Z \right|}_{\Sigma} \underbrace{  U_{2\delta-1}^*}_{V^*},
\end{equation}
where $Z$ is the diagonal matrix in $\mathbb{C}^{D \times D}$, which contains all $z_{k,j}$ for $k \in [d], j \in [2\delta-1]$ on its main diagonal and the operations $\sign$ and $|\cdot|$ are taken entrywise.

With this expression for the SVD at hand, Proposition \ref{theorem_0_eigenvalues} directly yields Corollary \ref{theo:operator non invertible}, as every vanishing $z_{k,j}$ corresponds to a zero singular value of $M$. Therefore, the matrix $M$ cannot be invertible and, consequently, neither can the operator $\mathcal{A}|_{T_{\delta}(\mathcal{H}^d)}$, as stated in Corollary \ref{theo:operator non invertible}.\\
Moreover, Proposition \ref{theorem_0_eigenvalues} provides explicit expressions for the index set $\mathcal{I}$ corresponding to these zero singular values of $M$, as we will see below.

\begin{proof}[Proof of Proposition \ref{theorem_0_eigenvalues}]
We first consider the case $1 \leq j \leq \delta$. 
By \eqref{def_J_k}, 
\begin{align*}
z_{k,j} & = (F^*_{2\delta-1} \cdot J_k)_{j,j} = \sum_{\l=1}^{\delta -j +1} \overline{w}_\l w_{j+\l-1} e^{\frac{2 \pi i (k-1) (\l-1)}{d}}.
\end{align*}

Setting $k = \frac{d}{2} +1$, we get
\begin{align*}
z_{\frac{d}{2} +1,j} 
& = \sum_{\l=1}^{\delta -j +1} \overline{w}_\l w_{j+\l-1} e^{ \pi i (\l-1) } 
= \sum_{\l=0}^{\delta -j} \overline{w}_{\l+1} w_{j+\l} (-1)^{\l} \\
& = \sum_{\l=0}^{\delta -j} w_{\delta -\l} \overline{w}_{\delta -j - \l +1}  (-1)^{\l} 
= \sum_{\l=0}^{\delta -j} w_{j+\l} \overline{w}_{\l+1} (-1)^{(\delta -j - \l)},
\end{align*}
where in the third equality we used the symmetry of the window $w$ and in the last one we changed the summation order.
We continue by averaging the first and last reformulation of $z_{\frac{d}{2} +1,j}$ obtaining
\begin{align*} \label{eq:proof_theo_sing_val1}
z_{\frac{d}{2} +1,j} 
&= \frac{1}{2} z_{\frac{d}{2} +1,j}  + \frac{1}{2} z_{\frac{d}{2} +1,j}  = \sum_{\l=0}^{\delta -j} w_{j+\l} \overline{w}_{\l+1} \frac{1}{2} \left( (-1)^\l + (-1)^{\delta -j - \l} \right) \nonumber \\
&= \sum_{\l=0}^{\delta -j} w_{j+\l} \overline{w}_{\l+1} \frac{(-1)^\l}{2} \left( 1 + (-1)^{\delta -j - 2\l} \right).
\end{align*}
When $\delta - j$ is odd, the last factor vanishes in all the summands and hence 
\[
z_{\frac{d}{2} +1,j} = 0.
\]
For the case $\delta +1 \leq j \leq 2\delta -1$, we get analogously
\begin{align*}
z_{\frac{d}{2} +1,j}   
&= 0 \text{ if $\delta - j$ is even.}
\end{align*}

\end{proof}
 Proposition \ref{theorem_0_eigenvalues} can be reformulated in terms of the matrix $V_2$ as follows.
\begin{corollary} \label{theo:missing entries}
Consider measurement masks of the form \eqref{def_general_mask}. Assume that $d$ is even and the window $w$ satisfies the symmetry condition
\begin{equation*}
w_n = \overline{w}_{\delta - n +1} \text{ for all } n \in [\delta].
\end{equation*}
Then the right singular vector $V^{(q)}$  corresponding to the singular value $|z_{\frac{d}{2}+1,r}|$
 is a column of $V_2$ provided  
the index $q$ is of the form
\begin{equation*} 
q = (2\delta -1) (d/2+1) + r
\end{equation*}
with
$r \in \mathcal{I}$, where the set $\mathcal I$ of size $ \left | \mathcal{I} \right| = \delta -1$ is given by
\begin{align*}
\mathcal{I} = \lbrace j \in [\delta]: \delta - j \text{ is odd}  \rbrace \cup \lbrace j \in [\delta-1]_{\delta+1}: \delta - j \text{ is even}  \rbrace
\end{align*}
Equivalently, $ \widehat{L^r}_{\frac{d}{2}+1}, r \in \mathcal{I},$ are not recovered in the step 2 of Algorithm \ref{algo:2}.
\end{corollary}

\subsection{Proof of Theorem~\ref{theo:diagonal_vec_relation} and Corollary \ref{cor:step2}}



\begin{proof}[Proof of Theorem \ref{theo:diagonal_vec_relation}]
First we use the definition of the DFT \eqref{def_discrete_fourier_trans} and the diagonals $L^r$ \eqref{def2_L_k} to obtain
\begin{align*}
\widehat{L^{r}}_{\xi} &= \displaystyle\sum_{n=1}^d e^{-\frac{2 \pi i (n-1) (\xi-1)}{d} } L^r_n \\
&= \begin{cases}
	\displaystyle\sum_{n=1}^d e^{-\frac{2 \pi i (n-1) (\xi-1)}{d} } \overline{(x_0)}_n (x_0)_{n+r-1} & \text{if  $ \; 1 \leq r \leq \delta$,}\\
	\displaystyle\sum_{n=1}^d e^{-\frac{2 \pi i (n-1) (\xi-1)}{d} } \overline{(x_0)}_{n+1} (x_0)_{n+1+r-2\delta}  & \text{if  $ \; 1+\delta \leq r \leq 2\delta -1$.}\\
    \end{cases}
\end{align*}
To reformulate the right hand side of \eqref{eq:theo_diagonal_vec_relation}, we need an alternative representation of the vectorization of $X_0$ as given in the following lemma.

\begin{restatable}{lemma}{ixo}
\label{lem:indexX0}
Decomposing $q\in [D]$ as 
\[
q=(2\delta-1)(\xi-1)+r \text{ with } \xi \in [d] \text{, } r\in [2\delta -1],
\]
one can express the vectorization of $X_0$ from formula \eqref{def_vec_X0} as
\begin{equation*}
  \vect(X_0)_q =
    \begin{cases}
      \overline{(x_0)}_{\xi} (x_0)_{\xi+r-1} 
& \text{if  $ \; 1 \leq r \leq \delta$,}\\
      \overline{(x_0)}_{\xi+1} (x_0)_{\xi+r-2\delta +1} & \text{if  $ \; 1+\delta \leq r \leq 2\delta -1$.}\\
    \end{cases}
\end{equation*}

\end{restatable}
The proof of the Lemma can be found in the Appendix.

Recall that by \eqref{def_svd_fourier_measurements}, one has that $V = U_{2 \delta -1}$. Thus, as for every column of $U_{2\delta-1}$,  the indices of $V^{(q)}$ can be expressed as
\begin{equation*}\label{def_U_2} 
\left(V^{(q)}\right)_\l
= \left( U_{2\delta -1}^{(q)} \right)_\l 
= \begin{cases}
\frac{1}{\sqrt{d}}  e^{\frac{2 \pi i (\xi-1) s}{d}} & \text{ if } \l~\mod (2\delta-1)=r,\\
0 & \text{ else },
\end{cases}      
\end{equation*}
where $s=\left\lfloor \frac{\l-1}{2\delta -1} \right\rfloor$. We obtain that
\begin{align*}
\sqrt{d} \; \langle \vect(X_0), V^{(q)} \rangle 
&= \displaystyle\sum_{n=1}^d e^{-\frac{2 \pi i (\xi-1) (n-1)}{d} } \vect(X_0)_{(2 \delta -1)(n-1)+r} \\
&= \begin{cases}
	\displaystyle\sum_{n=1}^d e^{-\frac{2 \pi i (n-1) (\xi-1)}{d} } \overline{(x_0)}_n (x_0)_{n+r-1} & \text{if  $ \; 1 \leq r \leq \delta$,}\\
	\displaystyle\sum_{n=1}^d e^{-\frac{2 \pi i (n-1) (\xi-1)}{d} } \overline{(x_0)}_{n+1} (x_0)_{n+1+r-2\delta} & \text{if  $ \; 1+\delta \leq r \leq 2\delta -1$,}\\
    \end{cases}
\end{align*}
where the last equality follows from Lemma \ref{lem:indexX0}.
\end{proof}
Corollary \ref{cor:step2} now follows directly from the definition of $X_\mathcal{S}$.
\begin{proof}[Proof of Corollary \ref{cor:step2}]
Combining the definition of $X_\mathcal{S}$ with \eqref{eq: pseudo projection}, we obtain
\[
\vect( X_{\mathcal{S}})  = M_{\mathcal{S}}^{-1} y =  M_{\mathcal{S}}^{-1} M \vect(X_0) = V_1 V_1^* \vect(X_0)
\] 
and 
\begin{align*}
\sqrt{d} \; \langle \vect(X_{\mathcal{S}}), V^{(q)} \rangle 
&= \sqrt{d} \langle V_1 V_1^* \vect(X_0) , V^{(q)} \rangle 
=  \sqrt{d} \langle  \vect(X_0) , V_1 V_1^* V^{(q)} \rangle \\
&= 
\begin{cases}
\sqrt{d} \langle  \vect(X_0) , V^{(q)} \rangle & \text{ if  $V^{(q)}$ column of } V_1, \\
0  & \text{ if  $V^{(q)}$ column of } V_2,
\end{cases}
\end{align*}
which concludes the proof.
\end{proof}

\subsection{Recovery of lost coefficients and proof of Theorem~\ref{theo:linear system}}\label{sec: recovery}
In this section we will show that under the assumptions of Theorem~\ref{theo:linear system}, Algorithm \ref{algo:2} provides a feasible and reliable procedure to complete the restricted low-rank matrix $X_0$. We first prove Lemma~\ref{lemma_represent_Ll_by_Lk}, which establishes that, under the assumptions that the algorithm is tractable and the solution of \eqref{assertion_1_lemma_represent_Ll_by_Lk} is unique, Algorithm \ref{algo:2} yields the right answer.



\begin{proof}[Proof of Lemma \ref{lemma_represent_Ll_by_Lk}]
Observe that for arbitrary $j \in [d]$  
\begin{align*}
\left( L^r \circ S_{\l-1} (L^r)^* \right)_j &= \overline{(x_0)}_j (x_0)_{j+r-1} \cdot (x_0)_{j+\l-1} \overline{(x_0)}_{j+\l+r-2} \\
&=  \overline{(x_0)}_j (x_0)_{j+\l-1} \cdot (x_0)_{j+r-1}  \overline{(x_0)}_{j+\l+r-2} = \left( L^{\l} \circ S_{r-1} (L^{\l})^* \right)_j.
\end{align*}
\end{proof}

In the remainder of this section, we will show tractability and uniqueness, which will both follow from the fact that under Assumptions  \textbf{(A1)} and  \textbf{(A2)}, \eqref{assertion_1_lemma_represent_Ll_by_Lk} is in fact a linear relation. 

The main idea is that for the index $r$ corresponding to the diagonal that is fully known by \textbf{(A1)}, the left hand side of \eqref{assertion_1_lemma_represent_Ll_by_Lk} is known, and all but one Fourier coefficients of the right hand side consists of products of different coefficients. Thus, one obtains linear equations in the real and imaginary parts of the single unknown coefficient. This is formalized in the following proof of Theorem~\ref{theo:linear system}.

\begin{proof}[Proof of Theorem \ref{theo:linear system}]
Let us denote by $L^r$ the diagonal with all known entries, provided by assumption \textbf{(A1)}. Due to assumption \textbf{(A2)} we are missing at most one Fourier coefficient on each of the other diagonals $\widehat{L^{\l}}$, say at position $q \in [d]$. 
To estimate this missing coefficient, we work with the discrete Fourier transform of the quadratic relationship \eqref{assertion_1_lemma_represent_Ll_by_Lk}. Using the convolution theorem \eqref{convolution_theorem} and Lemma \ref{lem:EFTP}, we obtain the following expression for the $j$-th Fourier coefficient.

\begin{align*} 
c_j := &
\left({L^{r} \circ S_{\ell-1} }(L^{r})^*\right)\widehat\ _j 
= \left(L^{\l} \circ S_{r-1} (L^{\l})^*\right)\widehat\ _j \\
& = \left( \widehat{L^{\l}} * (W_r \widehat{(L^{\l})^*)} \right)_j 
= \sum_{p=1}^d  \widehat{L^{\l}}_{j-p+1} \left( W_r \widehat{(L^{\l})^*} \right)_p \\
&= \sum_{p=1}^d \widehat{L^{\l}}_{j-p+1} e^{\frac{2\pi i (r-1) (p-1)}{d}} \widehat{(L^{\l})^*}_p = \sum_{p=1}^d  e^{\frac{2\pi i (r-1) (p-1)}{d}}  \widehat{L^{\l}}_{j-p+1}  \overline{\widehat{L^{\l}}}_{d-p+2}.
\end{align*}


We note that for $p \in \mathcal{C} := \left\lbrace j+1-q, d+2-q \right\rbrace$ the corresponding summand on the right hand side contains  $\widehat{L^{\l}}_q$ or $\overline{ \widehat{L^{\l}}}_q$; all other summands are known by Assumption \textbf{(A2)}. For $j=1$, $\mathcal{C}$ is a single index, which leads to a quadratic term $|\widehat{L^{\l}}_q|^2$. For this reason we skip the case $j=1$ in the following considerations. For $j\neq 1$, let 
\[
z_j:= \sum_{\substack{p=1 \\ p \notin \mathcal{C} }}^d e^{\frac{2\pi i (r-1)(p-1)}{d}} \widehat{L^{\l}}_{j-p+1} \overline{\widehat{L^{\l}}}_{d-p+2}.
\]
We obtain
\begin{align*}
c_j - z_j
& = \widehat{L^{\l}}_{q}  \underbrace{e^{\frac{2\pi i (r-1) (j-q)}{d}} \overline{\widehat{L^{\l}}}_{d-j+q+1}}_{a_j:=} 
+  \underbrace{e^{\frac{2\pi i (r-1) (d-q+1)}{d}} \widehat{L^{\l}}_{j-d+q-1}}_{b_j:=}  \overline{\widehat{L^{\l}}}_{q}.
\end{align*}

By decomposing the complex numbers in the previous equation into their real and imaginary part we get


\begin{align*}
\underbrace{\left[ \begin{array}{rr} Re(a_j) + Re(b_j) & - Im(a_j) + Im(b_j) \\ Im(a_j) + Im(b_j) & Re(a_j) - Re(b_j) \end{array}\right]}_{=: Q_j} \cdot \left[ \begin{array}{r} Re(\widehat{L^{\l}}_{q} ) \\ Im(\widehat{L^{\l}}_{q} ) \end{array} \right] = \underbrace{\left[ \begin{array}{r}  Re (c_j -z_j) \\ Im (c_j -z_j) \end{array} \right]}_{v_j}.
\end{align*}

Combining these linear systems for all $j=2, \dots, d$ we obtain the overdetermined linear system

\begin{align}\label{linear_system_to_be_solved}
\left( \begin{array}{r} Q_2 \\ Q_3 \\ \vdots \\ Q_d \end{array}\right) \cdot \left[ \begin{array}{r} Re(\widehat{L^{\l}}_{q} ) \\ Im(\widehat{L^{\l}}_{q} ) \end{array} \right] &= \left( \begin{array}{r} v_2 \\ v_3 \\ \vdots \\ v_d \end{array} \right) = Q \cdot \left[ \begin{array}{r} Re(\widehat{L^{\l}}_{q} ) \\ Im(\widehat{L^{\l}}_{q} ) \end{array} \right] = v.
\end{align}

This system can be solved using least squares to approximate the missing component $\widehat{L^{\l}}_q$. After this estimation step all entries of the diagonal $\widehat{L^{\l}}_q$ are available and one can obtain an estimate for the vector $L^{\l}$ via the inverse DFT. This procedure is repeated for all diagonals $L^\l$ with $1 \leq \l \leq \delta $ that have missing Fourier coefficients. This recovers the upper triangular part of $X_0$. To obtain the remaining entries, we use that $X_0$ is a Hermitian matrix.
\end{proof}

\section{Conclusion and Future work}

In this paper, we proposed a subspace completion technique, which extends the range of applicability of the \textit{BlockPR} algorithm in ptychography to a larger class of windows. Furthermore, our technique can be used as a regularizer for better noise robustness. The next step will be to analyze the case of more than one zero entry per diagonal in more detail. As this scenario gives rise to nonlinear dependencies, this will likely require a very different set of tools. Also, we plan to consider the extension of the technique to models with shifts longer than 1 as proposed in \cite{Preskitt.} and \cite{Melnyk.7820197122019}.

\section{Acknowledgements}

The authors thank Frank Filbir for many inspiring discussions on the topic.
FK acknowledges support by the German Science Foundation DFG in the context of an Emmy Noether junior research group (project KR 4512/1-1). OM has been supported by the Helmholtz Association (project Ptychography 4.0).

\bibliography{ms.bbl} 

\appendix

\section{Proof of Lemma~\ref{lem:indexX0} }
We start by considering the case that $1 \leq r \leq \delta$. Then to rewrite the index of the first factor on the right hand side of \eqref{def_vec_X0}, we observe that it follows from the decomposition of $q$ that
\begin{align*}
\frac{q+\delta -1}{2 \delta -1} = \frac{(2\delta -1)(\xi-1)+r+\delta -1}{2 \delta -1} = \xi - 1 + \frac{r+\delta-1}{2 \delta -1}.
\end{align*}
Since $r \in [\delta]$, it follows that
\begin{align*}
\frac{1}{2}<\frac{\delta}{2\delta -1} \leq \frac{r+\delta -1}{2 \delta -1} \leq \frac{2\delta -1 }{2\delta -1} = 1,
\end{align*}
and, thus,
\begin{align*}
\left\lceil \frac{q+\delta -1}{2 \delta -1} \right\rceil = \xi.
\end{align*}
Furthermore, we have 
\begin{align*}
(q+\delta -2) \mod(2\delta-1) &= ((2\delta -1)(\xi-1) +r+\delta-2) \mod(2\delta-1) \\
&= (r+\delta -2) \mod(2\delta -1) = r + \delta -2.
\end{align*}
Thus, the index of the second factor on the right hand side of \eqref{def_vec_X0} can be rewritten as
\begin{align*}
\left\lceil \frac{q+\delta -1}{2 \delta -1} \right\rceil + \left[ (q+\delta -2) \mod(2\delta -1) \right] -\delta +1 
= \xi+r -1.
\end{align*}
Consequently, we obtain that $\vect(X_0)_q$ defined in \eqref{def_vec_X0} can be written as
\begin{align*}
\vect(X_0)_q = \overline{(x_0)}_{\xi} (x_0)_{\xi+r-1} 
\text{ if } 1 \leq r \leq \delta \text{ and } \xi\in [d]. 
\end{align*}
In the \textbf{second case} we have $1+\delta \leq r \leq 2\delta -1$, which immediately gives\\
\begin{align*}
1 < \frac{2\delta}{2 \delta -1} \leq \frac{r+\delta -1}{2 \delta -1} \leq \frac{3 \delta -2}{2 \delta -1} \le 2,
\end{align*}
and, thus, the index of the first factor on the right hand side of \eqref{def_vec_X0} is given by
\begin{align*}
\left\lceil \frac{q+\delta-1}{2\delta -1} \right\rceil = \xi+1.
\end{align*}
Furthermore, analogous to the previous case we get
\begin{align*}
(q+\delta-2) \mod(2\delta-1) & = r-(\delta +1),
\end{align*}
so that the index of the second factor in \eqref{def_vec_X0} becomes
\begin{align*}
\left\lceil \frac{q+\delta -1}{2 \delta -1} \right\rceil + \left[ (q+\delta -2) \mod(2\delta -1) \right] -\delta +1 
&= \xi+r-2\delta +1.
\end{align*}
Finally, we get that $\vect(X_0)_q$ as defined in \eqref{def_vec_X0} can be written equivalently as
\begin{align*}
\vect(X_0)_q = \overline{(x_0)}_{\xi+1} (x_0)_{\xi+r-2\delta+1} 
\text{ if } 1+\delta \leq r \leq 2\delta-1 \text{ and } \xi\in [d].
\end{align*}

\qed

%
\end{document}